\documentclass[12pt]{amsart}
\usepackage{amsmath,amsfonts,amsthm,amsopn,cite,mathrsfs}
\usepackage{epsfig,verbatim}
\usepackage{subfigure}

\newcommand{\up}{uncertainty principle}

\newcommand{\tf}{time-frequency}

\newcommand{\onb}{orthonormal basis}

\newcommand{\bdl}{band-limited}

\newtheorem{tm}{Theorem}
\newtheorem{lemma}[tm]{Lemma}
\newtheorem{prop}[tm]{Proposition}

\newcommand{\rem}{\noindent\textsl{REMARK:}}

 \theoremstyle{definition}

\newcommand{\beqa}{\begin{eqnarray*}}
\newcommand{\eeqa}{\end{eqnarray*}}

\DeclareMathOperator*{\supp}{supp}

\newcommand{\field}[1]{\mathbb{#1}}
\newcommand{\bR}{\field{R}}        
\newcommand{\bN}{\field{N}}        
\newcommand{\bZ}{\field{Z}}        
\newcommand{\bP}{\field{P}}        
 \def\cF{\mathcal{F}}              

 \def\cH{\mathcal{H}}
 \def\cB{\mathcal{B}}

 \def\cO{\mathcal{O}}
 
\def\cP{\mathcal{P}}

 \def\cX{\mathcal{X}}

\def\rd{\bR^d}

\def\zd{\bZ^d}

\def\lrd{L^2(\rd)}

\def\intrd{\int_{\rd}}

\def\<{\left<}
\def\>{\right>}

\def\inv{^{-1}}

\def\mv1{M_v^1}

\newcommand{\norm}[1]{\lVert#1\rVert}

\newcommand{\vs}{\vspace{3 mm}}

\def\E{{{\Bbb E}\,}}
\def\P{{\Bbb P}}
\newcommand{\pn}{\cP _N}

\def\brd{\cB (R,\delta )}

\begin{document}
\begin{abstract}
We study the random sampling of band-limited functions of several
variables. If a \bdl\ function with bandwidth  has its essential
support  on a cube of volume
$R^d$, then $\cO (R^d \log R^d)$ random samples suffice to  approximate
the function  up to a given error with high probability.  
\end{abstract}

\title{Relevant Sampling of Band-limited Functions}
\author{Richard F.\ Bass}
\address{Department of Mathematics \\ University of Connecticut \\
Storrs CT 06269-3009 USA}
\email{r.bass@uconn.edu}
\author{Karlheinz Gr\"ochenig}
\address{Faculty of Mathematics \\
University of Vienna \\
Nordbergstrasse 15 \\
A-1090 Vienna, Austria}
\email{karlheinz.groechenig@univie.ac.at}
\date{}
\thanks{K.\ G.\ was
  supported in part by the  project P22746-N13  of the
Austrian Science Foundation (FWF)}
\thanks{R.\ B.\ was
   supported in part by NSF grant DMS-0901505}

\maketitle

\section{Introduction}

The nonuniform sampling of band-limited functions of several variables
remains a  challenging problem. Whereas in dimension $1$ the density
of a set essentially characterizes sets of stable sampling~\cite{OS02}, in higher
dimensions the density is no longer a decisive 
property of sets of stable  sampling. Only  a few strong and explicit
sufficient conditions are known, e.g.,~\cite{beurling66,MM10,OU08}. 

This  difficulty   is one of the reasons for taking a probabilistic approach to the
sampling problem~\cite{BG10,SZ04}. At first glance, one would guess that every
reasonably homogeneous set of points in $\rd $ satisfying Landau's
necessary density condition will generate a set of stable sampling. 
This intuition is far from true. To the best of our knowledge,  every
construction  in the literature of sets of random points in $\rd $ contains either
arbitrarily large holes with positive probability or concentrates near
the zero manifold  of a band-limited function. Both properties are
incompatible with a sampling inequality. See~\cite{BG10} for a
detailed discussion.

The difficulties  with the probabilistic approach lie in the
unboundedness of the configuration space $\rd $ and the infinite
dimensionality of the space of \bdl\ functions. To resolve this issue, we
argued in~\cite{BG10} that usually one observes  only finitely many samples of a
band-limited function and that these observations are drawn from a bounded subset of
$\rd $. Moreover,  since it does not make sense to sample a given
function $f$  in a region
where $f$ is small, we  proposed to  sample $f$  only on  its 
essential support. Since  $f$ is sampled  only in the relevant region,
this method might be called the ``relevant sampling of band-limited
functions.''  In this paper we continue our investigation of the
random sampling of band-limited functions and settle a question that
was left open in~\cite{BG10}, namely how many  random samples are 
required to approximate a band-limited function locally  to within a given
accuracy? 

To fix terms, recall that 
the space of \bdl\ functions is defined to be 
$$
\cB  = \{ f \in \lrd : \supp \hat{f} \subseteq [-1/2,1/2]^d\} \, ,
$$
where we have normalized the spectrum to be the unit cube and the
Fourier transform is normalized as $\hat{f}(\xi ) = \intrd f(x)
e^{-2\pi i x\cdot \xi } \, dx $. 
A set $\{x_j : j\in J\}\subseteq \rd $   is called a set of stable sampling or simply a
set of sampling~\cite{landau67}, if there exist constants $A,B>0$,
such that a    \emph{sampling inequality} holds: 
\begin{equation}
  \label{eq:1}
A\|f\|_2^2 \leq \sum _{j} |f(x_j)|^2 \leq B
\|f\|_2^2, \qquad \forall  f \in \cB \, .
\end{equation}

Next, we sample only on  the
essential support of $f$. Therefore we
let $C_R=[-R/2,R/2]^d$ and define the
subset 
$$\brd =\left\{ f \in \cB :  
    \,\, \int _{C_R} |f(x)|^2 \, dx \geq     (1-\delta )
    \|f\|_2^2 \right\}.$$

As a  continuation of ~\cite{BG10},  we will prove the following sampling
theorem.

\begin{tm}
  \label{main0}
 Let $\{x_j: j\in \bN \}$ be a sequence of
  independent and identically distributed  random variables that are
  uniformly distributed in $C_R$. 
   Suppose that $R\ge 2$,  that $\delta \in (0,1)$ and $\nu \in (0,1/2) $  are small enough, and that $0<\epsilon
   <1$.  
There exists a constant $\kappa$ so that if the number of samples $r$ satisfies 
      \begin{equation}
        \label{eq:hm6a}
        r \geq R^d \frac{1+\nu /3}{\nu ^2} \log
        \frac{2R^d}{\epsilon} \, ,
      \end{equation}
then the sampling inequality 
  \begin{equation}
    \label{eq:hm4a}
\frac{r}{R^d} \Big( \tfrac{1}{2} - \delta  -\nu -12 \delta \kappa
\Big)  \|f\|_2^2 \leq \sum _{j=1}^r |f(x_j)|^2 \leq r\|f\|_2^2  \qquad
\text{ for all } f\in \brd     
  \end{equation}
holds with probability at least $1-\epsilon $. The constant  $\kappa $ can be
taken to be $\kappa =e^{d\pi }$.    
\end{tm}

The formulation of Theorem~\ref{main0} is similar to
~\cite[Thm.~3.1]{BG10}. The main point is that 
only $\cO (R^d \log R^d )$ samples are required for a
sampling inequality to hold with high probability. In~\cite{BG10} we
used a metric entropy argument to show that $\cO (R^{2d})$ samples suffice.
We expect that the order $\cO (R^d \log R^d )$ is optimal.
We point out that in addition all constants are now
explicit.

Our idea is to replace the sampling of \bdl\ function in $\brd$ by a
finite-dimensional problem, namely the sampling of the corresponding
span of prolate spheroidal functions on the cube $[-R/2,R/2]^d$ and
then use error estimates. For the probability estimates we use a new
tool, namely the powerful  matrix Bernstein inequality of Ahlswede and
Winter~\cite{AW02} in the optimized version of Tropp~\cite{Tro11}. 

The remainder of the paper contains the analysis of a related
finite-dimensional problem for prolate spheroidal functions in
Section~2 and transition to the infinite-dimensional problem in $\brd
$ with the necessary error estimates  in Section~3. The appendix
contains an elementary estimate for the constant $\kappa $.

\section{Finite-Dimensional Subspaces of $\cB $}

We first study a sampling problem in a finite-dimensional subspace
related to the set $\brd $.

\textbf{Prolate Spheroidal Functions.} 
Let $P_R$ and $Q$ be the projection operators defined by 
\begin{equation}
  \label{eq:r2}
  P_Rf = \chi _{C_R} f \qquad \mathrm{and} \qquad Qf = \cF \inv (\chi
  _{[-1/2,1/2]^d} \hat{f} ) \, ,
\end{equation}
where $\cF^{-1}$ is the inverse Fourier transform.
The composition of these orthogonal  projections
\begin{equation}
  \label{eq:r3}
  A_R =  Q P_R Q
\end{equation}
is the operator of time and frequency limiting.
This
operator arises frequently in the context of \bdl\ functions and \up
s. The localization operator $A_R$ is a compact positive  operator of trace
class,
and by results of Landau, Slepian, Pollak, and
Widom~\cite{LP61,LP62,SP61,slep64,widom64} the eigenvalue
distribution spectrum is precisely known.
 We summarize the properties of the spectrum that we will need.

Let $A_R ^{(1)}$ denote the operator of \tf\ limiting in dimension
$d=1$. This operator   can be  defined explicitly on $L^2 (\bR ) $ by the  formula
$$
(A_R^{(1)} f)\,\widehat{} \,(\xi ) = \int _{-1/2} ^{1/2} \frac{\sin \pi R (\xi
  - \eta )}{\pi (\xi -\eta)} \hat{f}(\eta) \, d\eta \qquad \text{ for
} |\xi| \leq 1/2 \, .
$$
The  eigenfunctions of $A_R^{(1)}$ are the prolate spheroidal
functions and let the corresponding eigenvalues $\mu _k = \mu _k  (R)$  be arranged  in
decreasing order. According to ~\cite{landau93} they satisfy 
\begin{align*}
  & 0 < \mu _k(R) <1 \qquad \forall k\in \bN, \\
  & \mu _{[R]+1}(R)  \leq 1/2 \leq \mu _{[R]-1}(R) \, ;
\end{align*}
  As a consequence any  function with spectrum $[-1/2,1/2]$
and ``essential'' support on $[-R/2,R/2]$  is close to the span of the
first $R$ prolate spheroidal functions.  In particular, we may think of  $\brd $ as, roughly,  almost a subset of
a finite-dimensional space of dimension $R$.

 The \tf\ limiting operator $A_R $ on $\lrd $
 is the $d$-fold tensor product of  $A_R ^{(1)}$, $A_R =
A_R^{(1)} \otimes \dots \otimes A_R^{(1)}$. Consequently, $\sigma(A_R)$,  the
spectrum of $A_R$, is 
$$
\sigma (A_R)  = \{ \lambda \in (0,1) : \lambda = \prod 
_{j=1}^d \mu _{k_j}, \mu_{k_j} \in \sigma (A_R^{(1)}) \} \, .
$$
Since $0<\mu  _k<1 $, $A_R$ possesses at most $R^d$ eigenvalues
greater than or equal to $1/2$. 
Again we arrange the eigenvalues of $A_R$ by magnitude $1> \lambda
_1 \geq \lambda _2 \geq \lambda _3 \dots \geq \lambda _n \geq \lambda
_{n+1} \geq \dots  >0$. 
Let $\phi_j$ be the eigenfunction corresponding to $\lambda_j$.

We fix $R$ ``large'' and $\delta \in (0,1)$. Let 
$$
\cP _N = \mathrm{span}\, \{ \phi _j: j=1, \dots , N\}
$$
be  the span of the first $N$ eigenfunctions of the time-frequency
limiting operator $A_R$ (one might call them ``multivariate prolate
polynomials''). For properly chosen $N$,
 $\cP _N$ consists of functions in $\brd $. See
Lemma~\ref{qestim}. 

By Plancherel's theorem,
$$\langle Qf,g\rangle=\langle\chi_{[-1/2,1/2]^d} \hat f,\hat g\rangle
=\langle \hat f, \chi_{[-1/2,1/2]^d}\hat g\rangle
=\langle f,Qg \rangle.$$
Then for $f\in \cB$ we have $Qf=f$, and so
\begin{equation}\label{AR-eq}
\langle A_R f,f \rangle =\langle P_RQf, Qf\rangle=\langle P_Rf,f\rangle
=\int_{C_R} |f(x)|^2\, dx.
\end{equation}

We first study random sampling in the finite-dimensional space $\cP
_N$. In the following $\|f\|_{2,R}$ denotes the normalized $L^2$-norm
of $f$ restricted to the cube $C_R = [-R/2,
  R/2]^d$:
$$
\|f\|_{2,R}^2 =  \int _{C_R} |f(x)|^2 \, dx  \, .
$$

\begin{prop}\label{prop1}
   Let $\{x_j: j\in \bN \}$ be a sequence of
  independent and identically distributed
 random variables that are uniformly distributed in $[-R/2,
  R/2]^d$. Then 
 \begin{align}
    \bP\Bigg(\inf _{f\in \cP _N , \|f\|_{2} = 1  }
    \frac{1}{r}\sum_{j=1}^r (|f(x_j)|^2 &-\frac{1}{R^d} \|f\|_{2,R}^2)
    \leq - \frac{\nu}{R^d} \Bigg) 
    \label{eq:n1}\\
&\leq 
N \exp\Bigg(- \frac{\nu^2 r}{R^{d}(1+\nu /3)}\Bigg)\nonumber
\end{align}
for $r \in \bN $ and $\nu \geq 0$. 
\end{prop}

\vs

\noindent \emph{Proof.}
We prove the proposition in several steps. First, 
  since $\cP _N$ is finite-dimensional, the sampling inequality for $\cP
_N$ amounts to a statement about the spectrum of an underlying
(random) matrix.

Let $f = \langle c,\phi\rangle=\sum _{k=1}^N c_k \phi _k\in \cP _N$, so that $|f(x_j)| ^2 =
\sum _{k,l = 1}^N c_k \overline{c_l} \phi _k(x_j) \overline{\phi
  _l(x_j) }$. Now define the $N\times N$  matrix  $T_j$ of rank one by
letting the $(k,l)$ entry be 
\begin{equation}
  \label{eq:hm2}
(T_j)_{kl} = \phi _k(x_j) \overline{\phi
  _l(x_j) } \, .  
\end{equation}
 Then $|f(x_j)|^2 = \langle c, T_j c\rangle$. Since each
random variable $x_j$ is uniformly distributed over $C_R$ and $\phi
_k$ is the $k$-th eigenfunction of the localization operator $A_R$,
using \eqref{AR-eq}  the
expectation of the $kl$-th entry is 
\begin{align}
  \E \Big( (T_j)_{kl} \Big) &= \frac{1}{R^d} \int _{C_R} \phi _k(x) \overline{\phi_l
  (x) } \, dx = \frac{1}{R^d} \langle A_R \phi _k , \phi _l\rangle 
  \label{eq:n2}\\
&= \frac{1}{R^d} \lambda _k \delta _{kl} \qquad k,l = 1,
\dots , N,\nonumber
\end{align}
where $\delta_{kl}$ is Kronecker's delta.
Consequently the expectation of $T_j$ is the diagonal matrix
\begin{equation}
  \label{eq:n3}
  \E (T_j) = \frac{1}{R^d} \mathrm{diag}\, (\lambda _k) =:
  \frac{1}{R^d} \Delta \, .
\end{equation}
We may now rewrite the expression in~\eqref{eq:n1} as 
\begin{align}
\inf_{f\in \cP _N , \|f\|_{2} = 1  } & 
    \frac{1}{r}\sum_{j=1}^r \Big(|f(x_j)|^2 - \frac{1}{R^d} \|f\|_{2,R}^2\Big)
    \notag \\
&= \inf  _{\|c\|_2 = 1} \frac{1}{r} \sum _{j=1}^r (\langle c, T_j
c\rangle - \langle c,  \E (T_j)c\rangle ) \notag   \\
&= \lambda _{\mathrm{min}} \Big(\frac{1}{r} \sum _{j=1}^N  (T_j -
  \E (T_j)) \Big)\, \label{eq:n?}
\end{align}
where we use $\lambda _{\min } (U) $ for the smallest eigenvalue of a
self-adjoint matrix $U$. 

Consequently, we have to estimate  a  probability for  the matrix norm
of a sum of random matrices.  
We do this using a matrix Bernstein inequality due to 
 Tropp~\cite{Tro11}.
Let  $\lambda _{\max } (A)$ be the largest singular value of a matrix
$A$ so that $\|A\| \allowbreak = \lambda _{\max } (A^*A)^{1/2}$ is the operator norm
(with respect to the $\ell ^2$-norm).

\begin{tm}
{\rm (Tropp) }
  Let $X_j$ be a sequence of independent, random self-adjoint $N\times
  N$-matrices. Suppose that 
$$
\E X_j = 0 \qquad \text{ and } \qquad \|X_j \| \leq B \qquad a.s.
$$
and let
$$
\sigma ^2 = \Big\| \sum _{j=1}^r \E (X_j^2) \Big\| \, .
$$
Then for all $t\geq 0$, 
\begin{equation}
  \label{eq:1A}
  \P \Big(\lambda _{\max } \Big(\sum _{j=1}^r X_j\Big) \geq t\Big)  \leq N \exp \Big( -
  \frac{t^2/2}{\sigma ^2 +Bt/3} \Big) \, .
\end{equation}
\end{tm}

To apply the matrix Bernstein inequality, we set $X_j = T_j - \E
(T_j)$. We need to calculate $\|X_j\|$ and $\| \sum _j \E (X_j^2)
\|$. Clearly $\E (X_j) = 0$. 
\begin{lemma}\label{help1}
Under the conditions stated above we have   
\begin{align*}
  & \|X_j \| \leq 1, \\ 
  & \E (X_j^2)  \leq R^{-d} \Delta, \\
 \text{ and } \quad & \sigma ^2 = \|\sum _{j=1}^r \E (X_j^2) \| \leq \frac{r}{R^d}.
  \end{align*}
  \end{lemma}

  \begin{proof}
(i) To estimate  the matrix norm of $X_j$, recall that 
\begin{equation}
  \label{eq:n5}
|f(x)| \leq \|f\|_2 \qquad \forall f\in \cB \, .   
\end{equation}
Hence we obtain 
$$
  \|X_j \|  = \sup _{ \|f\|_2 = 1} \Big| \,  |f(x_j)|^2 - R^{-d}
  \|f\|_{2,R}^2 \Big| 
\leq \|f\|_\infty - R^{-d} \|f\|_{2,R}^2 \leq \|f\|_2 = 1 \, .
$$

(ii) Next we calculate the matrix $\E (X_j^2)$:  
\begin{align*}
  \E (X_j^2) &= \E (T_j^2) - R^{-d} \E (T_j\Delta ) - R^{-d} \E
  (\Delta T_j) + R^{-2d} \Delta ^2 \\
&=  \E (T_j^2) -  R^{-d} \E (T_j) \Delta  - R^{-d} \Delta \E
  ( T_j) + R^{-2d} \Delta ^2  = \E (T_j^2) -  R^{-2d} \Delta ^2\, .
\end{align*}
Furthermore, the
square of the rank one matrix $T_j$ is the (rank one) matrix
\begin{align*}
(T_j^2)_{km} & = \sum _{l=1}^N (T_j)_{kl}(T_j)_{lm}  \\
&= \sum _{l} \phi _k(x_j) \overline{\phi _l(x_j)}\phi
_l(x_j)\overline{\phi _m(x_j)} \\
&= \Big( \sum _{l=1}^N |\phi _l (x_j)|^2 \Big) \, (T_j)_{km} \, .
\end{align*}
Writing $m(x)=\sum_{l=1}^N |\phi_l(x)|^2$,  we obtain
\begin{equation}
  \label{eq:n4}
  T_j ^2= m(x_j) T_j \, , 
\end{equation}

Let $s$ be the function  whose Fourier transform is
given by $\hat{s} = \chi _{[-1/2,1/2]^d}$ and let $T_xf(t) = f(t-x)$ be
the translation operator. Then  it is well known that $T_xs$ is the reproducing kernel for $\cB$,
that is, 
$$f(x) = \langle f, T_xs\rangle .$$ 
 To see this, by Plancherel's theorem
 and the inversion formula for the Fourier transform, if $f\in \cB$,
$$\langle f, T_xs\rangle=\langle \hat f, e^{-2\pi i x\cdot \xi}\hat s\rangle 
 =\int_{[-1/2,1/2]^d} e^{2\pi i x\cdot \xi} \hat f(\xi) \, d\xi
=\int e^{2\pi i x\cdot \xi} \hat f(\xi) \, d\xi
=f(x).$$

Since the eigenfunctions $\phi _l$ form an \onb\ for $\cB $, the
factor $m(x_j)$  in \eqref{eq:n4} is majorized by 
$$
m(x_j) = \sum _{l=1}^N |\phi _l (x_j)|^2 = \sum _{l=1}^N |\langle \phi
_l,  T_{x_j} s \rangle |^2 \leq \sum _{l=1}^\infty  |\langle \phi
_l,  T_{x_j} s \rangle |^2  = \|T_{x_j} s\|_2^2 = 1 \, .
$$
Since  $ T_j^2 \leq T_j $ and the expectation preserves the cone of
positive (semi)definite matrices (see, e.g.~\cite{Tro11}), we have
$
\E (T_j^2) \leq \E (T_j) = R^{-d} \Delta \, , 
$
and $$\E (X_j^2) = \E (T_j^2) - R^{-2d} \Delta ^2 \leq R^{-d} \Delta
.$$ 
(iii) Now  the variance of the sum of positive (semi)definite random matrices is majorized
by 
$$
\sigma ^2 = \Big\|\sum _{j=1}^r \E (X_j^2) \Big\| \leq  \Big\|\sum _{j=1}^r \E
(T_j) \Big\| = \frac{r}{R^d} \|\Delta \| \leq \frac{r}{R^d} \, .
$$
  \end{proof}

\begin{proof}[End of the proof of Proposition~\ref{prop1}]  
Now we have all information to finish the proof of
Proposition~\ref{prop1}. 
Since $\lambda _{\min } (T) = - \lambda
_{\max} (-T)$, we substitute these estimates into the matrix Bernstein
inequality with $t= r\nu /R^d$, and
obtain that
$$
\E \Big( \lambda _{\min } \Big(\sum _{j=1}^r (T_j - \E (T_j))\Big)  \leq - r\nu /R^d \Big) \leq N \exp
\Big( - \frac{r^2\nu ^2 R^{-2d}}{rR^{-d} + r\nu R^{-d}/3}\Big)\, .
$$ 
Combined with \eqref{eq:n?}, the proposition is proved. 
        \end{proof}


Random matrix theory offers several methods to obtain probability
estimates for the spectrum of random matrices.  In~\cite{BG10} we 
used  the entropy
method. We also mention the influential work of
Rudelson~\cite{Rud99} and the recent papers~\cite{MP06,RV09} on random  matrices
with independent columns. The matrix Bernstein inequality  offers a new approach and makes the probabilistic
part of the argument almost painless. The matrix Bernstein inequality was first
derived in~\cite{AW02} and improved in several subsequent papers, in
particular in~\cite{Oliv10}. The   version with the best constants
is due to Tropp~\cite{Tro11}. Matrix Bernstein inequalities also  simplify
many probabilistic arguments in compressed sensing; see the
forthcoming book~\cite{FR12}.


\section{From Sampling of Prolate Spheroidal Functions to Relevant Sampling of Bandlimited Functions }

Let $\alpha $ be the value of the $N$-th eigenvalue of $A_R$, that is, $\alpha
= \lambda _N$, let $E= E_N$ be the
orthogonal projections from $\cB $ onto $\pn $, and let $F = F_N =
\mathrm{I} - E_N$. Intuitively, since $f \in \brd $ is essentially
supported on the cube $C_R$, it should be close to the span of  the largest
eigenfunctions of $A_R$ and thus $Ff$ should be small. The following
lemma gives a precise estimate. Compare also with the proof of ~\cite[Thm.~3]{LP62}. 


\begin{lemma} \label{qestim}
  If $f\in \brd $, then 
  \begin{align*}
 \|Ef\|_2^2 &\geq \Big(1-\frac{\delta}{1-\alpha}\Big)   \|f\|_2^2, \\
 \|Ef\|_{2,R}^2 &\geq \alpha \Big(1-\frac{ \delta}{1-\alpha}\Big)
  \|f\|_2^2, \\
\|Ff\|_2^2 &\leq \frac{\delta}{1-\alpha}
  \|f\|_2^2 \, .   
  \end{align*}
\end{lemma}

\begin{proof}
  Expand $f\in \cB $ with respect to the prolate spheroidal functions
  as $f= \sum _{j=1}^\infty c_j \phi _j$. Without loss of generality,
  we may assume that $\|f\|_2 = \|c\|_2 = 1$. Since $f\in \brd $, we
  have that
$$
1-\delta \leq \|f\|_{2,R}^2 = \int _{C_R} |f(t)|^2 \, dt = \langle A_R f,f\rangle  =
\sum _{j=1}^\infty |c_j|^2 \lambda _j \, .
$$
Set 
$$A = \|Ef\|_2^2 =  \sum _{j=1}^N |c_j|^2
$$
 and $B = \sum _{j>N} |c_j|^2 =  1-A = \|Ff\|_2^2$. 
Since $\lambda _j \leq \lambda _N = \alpha $ for $j>N$, we estimate
$A= \|Ef\|_2^2$ as follows:
\begin{align*}
  A&= \sum _{j=1}^N |c_j|^2 \geq  \sum _{j=1}^N |c_j|^2 \lambda _j \\
&= \sum _{j=1}^\infty  |c_j|^2 \lambda _j - \sum _{j=N+1}^\infty
|c_j|^2 \lambda _j \\
&\geq 1-\delta - \lambda _N \sum _{j=N+1}^\infty  |c_j|^2  \\
&= 1-\delta - \alpha (1-A) \, . 
\end{align*}
The inequality $A \geq 1-\delta  - \alpha (1-A)$ implies that 
$\|Ef\|_2^2 = A \geq 1 - \frac{\delta }{1-\alpha} $ and 
using the orthogonal decomposition $f=Ef+Ff$,
 $$
B = \|Ff\|_2^2  \leq \frac{\delta }{1-\alpha} \, .
$$ 
Finally, $\|Ef\|_{2,R}^2 = \sum _{j=1}^N \lambda _j |c_j|^2 \geq
\alpha A \geq \alpha (1 - \frac{\delta }{1-\alpha})$, as claimed. 
\end{proof}

\noindent\textsl{REMARK} (due to J.-L.\ Romero): As mentioned 
in~\cite{BG10}, if $f\in \brd$ and $f(x_j) = 0$ for sufficiently many
samples $x_j\in C_R$, then $f\equiv 0$. However, $f$ cannot be
completely determined by samples in $C_R$ alone. This is a consequence
of the fact that $\brd$ is not a linear space. Given a finite subset
$S \subseteq C_R$, consider  the
finite-dimensional subspace $\cH_0$  of $\cB$ spanned by the reproducing
kernels $T_{x}s, x\in S $. If $\phi \in \cH _0^\perp $, then $\phi (x)
= \langle \phi , T_x s\rangle = 0$ for $x\in S$. Thus by adding a
function in $\cH _0 ^\perp $ of sufficiently small norm to $f\in \brd
$, one obtains a different function with the same samples. More precisely,
let $f\in \brd$ with $ \|f\|_2=1 $ and $\int _{C_R} |f(x)|^2 dx = \gamma > 1-\delta$
and $\phi \in \cH _0^\perp $ with $\|\phi \|_2 =1$. Then $f(x) +
\epsilon \phi
(x) = f(x)$ for $x\in S$ and  $f+\epsilon \phi \in \brd $ for
sufficiently small $\epsilon >0$.  

\vs 

Despite this non-uniqueness, one can approximate $f$ from the samples
up to an accuracy $\delta $, as is shown by the next lemma.


We will require a standard estimate for sampled $2$-norms, a so-called
Plancherel-Polya-Nikolskij inequality~\cite{triebel83}. Assume that
$\cX = \{x_j\}\subseteq \rd $ is relatively separated, i.e., the ``covering index''  
$$ \max _{k\in \zd } \#
\cX \cap (k+[-1/2,1/2]^d) =: N_0 <\infty $$
 is finite. Then there exists
a constant $\kappa>0$, such
that  
\begin{equation}
  \label{eq:heu1}
  \sum _{j=1}^\infty  |f(x_j)|^2 \leq \kappa N_0 \|f\|_2^2   \qquad \text{ for all }
  f\in \cB \, .
\end{equation}
The constant $\kappa $ can be chosen as $\kappa = e^{d \pi }$. Since
the standard  proof in~\cite{triebel83} uses  a maximal inequality
with an non-explicit constant,  we will give a simple
argument using Taylor series in the appendix.

\begin{lemma}\label{approxrec}
Let $\{x_j: j=1, \dots , r\}$ be a finite subset of $C_R$ with
covering index $N_0$. Then the
solution to the least square problem 
\begin{equation}
  \label{eq:heu3}
p_{opt} = \mathrm{arg min} _{p\in \pn }  \Big\{ \sum _{j=1}^r |f(x_j) - p(x_j)|^2   \Big\}
\end{equation}
satisfies the error estimate
\begin{equation}
  \label{eq:heu2}
  \sum _{j=1}^r |f(x_j) - p_{opt}(x_j)|^2 \leq N_0 \kappa \frac{\delta
  }{1-\alpha} \|f\|_2^2 \qquad \text{ for all } f\in \brd \, .  
\end{equation}
\end{lemma}
\begin{proof}
We  combine Lemma~\ref{qestim} with~\eqref{eq:heu1}. 
\begin{align*}
  \sum _{j=1}^r |f(x_j) - p_{opt}(x_j)|^2 &\leq \sum _{j=1}^r |f(x_j)
  - Ef(x_j)|^2 \\
&= \sum _{j=1}^r |Ff(x_j)|^2 \leq \kappa N_0 \|Ff\|_2^2 \\
&\le \kappa N_0 \frac{\delta}{1-\alpha}\|f\|_2^2
\end{align*}
\end{proof}


Next we compare sampling inequalities for the space of prolate
polynomials $\pn $  to sampling inequalities for functions in $\brd
$.

\begin{lemma}\label{compprolo}
Let $\{x_j: j=1, \dots , r\}$ be a finite subset of $C_R$ with
covering index $N_0$. 

 If the inequality 
  \begin{equation}
    \label{eq:n6}
    \frac{1}{r} \sum _{j=1}^r \Big(|p(x_j)|^2 - R^{-d}
    \|p\|_{2,R}^2\Big)  \geq  -\frac{\nu}{R^d} \|p\|_2^2
  \end{equation}
 holds for all $p\in \pn $, then the inequality
 \begin{equation}
   \label{eq:n7}
        \sum _{j=1}^r |f(x_j)|^2 \geq  A  \|f\|_2^2
 \end{equation}
 holds for all $f\in \brd $ with a constant
$$
A  = \frac{r}{R^d} \Big(\alpha - \frac{\alpha \delta}{1-\alpha} -\nu \Big)  -  2 \kappa N_0
\frac{\delta }{1-\alpha}
$$
\end{lemma}

\rem\ For $A $ to be positive we need 
$$
r \geq  R^d  \, \frac{2 \kappa N_0
\frac{\delta }{1-\alpha}}{\alpha - \frac{\alpha \delta}{1-\alpha} -\nu}. $$

\begin{proof}
Using the triangle inequality and the  orthogonal decomposition $f= Ef + Ff$, we 
  estimate 
$$
    \Big(\sum _{j=1}^r |f(x_j)|^2 \Big)^{1/2} \geq  \Big(\sum _{j=1}^r
    |Ef(x_j)|^2 \Big)^{1/2} -  \Big(\sum _{j=1}^r |Ff(x_j)|^2
    \Big)^{1/2} \, .
$$
Taking squares and using \eqref{eq:heu1} on $Ef$ and $Ff$ in the cross product
term, we continue as 
\begin{align*}
\sum _{j=1}^r|f(x_j)|^2 & \geq \sum _{j=1}^r |Ef(x_j)|^2 - 2
\Big(\sum _{j=1}^r |Ef(x_j)|^2 \Big)^{1/2}   \Big(\sum _{j=1}^r |Ff(x_j)|^2 \Big)^{1/2} \\
&\qquad \qquad +\sum _{j=1}^r |Ff(x_j)|^2 \\
& \geq \sum _{j=1}^r |Ef(x_j)|^2 - 2 \kappa N_0
\|Ef\|_2 \|Ff\|_2 \\
&\geq   \sum _{j=1}^r |Ef(x_j)|^2 - 2 \kappa N_0
\frac{\delta }{1-\alpha} \|f\|_2^2 \, ,
  \end{align*}
since by  Lemma~\ref{qestim},  $\|Ff\|_2^2 \leq
\tfrac{\delta }{1-\alpha}\|f\|_2^2 $ and $\|Ef\|_2\leq \|f\|_2 $.  
Now we make use of  hypothesis
~\eqref{eq:n6} and Lemma~\ref{qestim} and obtain 
\begin{align*}
     \sum _{j=1}^r |f(x_j)|^2   &\geq  \sum _{j=1}^r |Ef(x_j)|^2 -  2 \kappa N_0
\frac{\delta }{1-\alpha} \|f\|_2^2 \\ 
&\geq \frac{r}{R^d}  \|Ef\|_{2,R}^2 - \frac{\nu r}{R^d} \|Ef\|_2^2  - 2 \kappa N_0
\frac{\delta }{1-\alpha} \|f\|_2^2 \\
&\geq \frac{\alpha r }{R^d} \Big( 1-\frac{\delta}{1-\alpha  } \Big)
\|f\|_{2}^2 - \frac{\nu r}{R^d}  \|f\|_2^2 -  2 \kappa N_0
\frac{\delta }{1-\alpha} \|f\|_2^2 \, .
 \end{align*}
So we may choose $A$ to be 
$$
A=  \frac{ r }{R^d} \Big( \alpha-\frac{\alpha\delta}{1-\alpha  } - \nu
\Big)
 -  2 \kappa N_0
\frac{\delta }{1-\alpha} \, . 
$$
 \end{proof}

\vs

The final ingredient we need is a deviation inequality for the
covering index $N_0 = \max _{k\in \zd  } \{x_j\} \cap (k+[-1/2,1/2]^d)$.

\begin{lemma}\label{N0estimate}
Suppose $R\ge 2$ and $\{x_j: j=1, \ldots, r\}$ are independent
and identically distributed random variables that are  uniformly
distributed  over $C_R$.
Let $a>R^{-d}$. Then 
$$\P(N_0>ar)\le (R+2)^d\exp\Big(-r\big(a\log(aR^d)-(a-R^{-d})\big) \Big).$$
\end{lemma}

\begin{proof} 
Let $D_k=k+[-1/2,1/2]^d$ for $k\in \zd$.  
Note that we need at most  $(R+2)^d$ of the $D_k$'s to cover $C_R$.
If $N_0>ar$, then for at least one $k$,  $D_k$ must contain at
least $ar$ of the $x_j$'s. Therefore
\begin{equation}\label{eq:N0est}
\P(N_0>ar)\le (R+2)^d \max_{k\in \zd} 
\P(\#\{x_j\}\cap D_k>ar).
\end{equation}

Fix $k\in \zd$. For any $b>0$,
by Chebyshev's inequality
\begin{align*}
\P(\#\{x_j\}\cap D_k>ar)&=\P\Big(\sum_{j=1}^r \chi_{D_k}(x_j)>ar\Big)
=\P\Big(\exp\Big(b \sum_{j=1}^r \chi_{D_k}(x_j)\Big)>e^{bar}\Big)\\
&\le e^{-bar}\E \exp\Big(b \sum_{j=1}^r \chi_{D_k}(x_j)\Big).
\end{align*}
Since the $x_j$
are  uniformly distributed  over $C_R$, then  $\chi_{D_k}(x_j)$ is equal to 1 with probability at most $R^{-d}$ and otherwise equals zero. Therefore, using
the independence,
\begin{align*}\P(\#\{x_j\}\cap D_k>ar)&\le e^{-bar}\prod_{j=1}^r \E
e^{b\chi_{D_k}(x_j)}\\
&\le e^{-bar}((1-R^{-d})+e^b R^{-d})^r =e^{-bar}((1+(e^b-1)R^{-d})^r\\
&\le e^{-bar} (\exp((e^b-1)R^{-d}))^r.
\end{align*}
With the optimal choice  
$b=\log (aR^d)$  the last term is then 
$$\exp\Big( -r\big(a\log(aR^d)-(a-R^{-d})\big) \Big).$$
Substituting this in \eqref{eq:N0est}
proves the lemma.
\end{proof}

By combining the finite-dimensional result of Proposition~\ref{prop1} with the
estimates of Lemmas~\ref{compprolo} and \ref{N0estimate} and the
appropriate choice of the free parameters, we obtain the following 
theorem. 
\begin{tm}
  \label{main}
 Let $\{x_j: j\in \bN \}$ be a sequence of
  independent and identically distributed  random variables that are uniformly distributed in $C_R$.
   Suppose $R\ge 2$,  
$$\delta < \frac{1 }{ 2( 1 + 12 \kappa )} \, ,$$
and
$$\nu<\tfrac12 -\delta(1+12\kappa).$$
Let 
\begin{equation}\label{eq:seetA}
A= \frac{r}{R^d} \Big( \tfrac{1}{2} - \delta  -\nu -12 \delta \kappa
\Big) \, .
\end{equation}
   Then the sampling
  inequality 
  \begin{equation}
    \label{eq:hm4}
A \|f\|_2^2 \leq \sum _{j=1}^r |f(x_j)|^2 \leq r \|f\|_2^2  \qquad
\text{ for all } f\in \brd     
  \end{equation}
holds with probability at least 
\begin{equation}
  \label{eq:hm5}
  1 - R^d \exp\Bigg(- \frac{\nu^2 r}{R^{d}(1+\nu /3)}\Bigg) -(R+2)^d
\exp\Big(-\frac{r}{R^d}(3\log 3-2)\Big)\, . 
\end{equation}
\end{tm}

\begin{proof}
Since $|f(x)|\le \norm{f}_2$ for $f\in \cB $, the right hand inequality
in \eqref{eq:hm4} is immediate.
We take  $\alpha  = 1/2$ and  $N=R^d$   in 
Proposition \ref{prop1} and $a=3R^{-d}$ in
Lemma~\ref{N0estimate}.  
Let $$V_1=\Big\{
\inf _{f\in \cP _N , \|f\|_{2} = 1  }
    \frac{1}{r}\sum_{j=1}^r (|f(x_j)|^2 -\frac{1}{R^d} \|f\|_{2,R}^2)
    \leq - \frac{\nu}{R^d}
\Big\} $$
and let 
$$V_2=\{N_0>ar\}.$$

By Proposition \ref{prop1} and Lemma \ref{N0estimate}, 
the probability of $(V_1\cup V_2)^c$ is bounded below by \eqref{eq:hm5}.
By Lemma \ref{compprolo},
$$\frac1{r}\sum_{j=1}^r|f(x_j)|^2\ge A\norm{f}_2^2$$
for all $f\in \brd$ on the set $(V_1\cup V_2)^c$.
With $\alpha = 1/2$ and $N_0 = 3R^{-d}$ the lower bound  $A$ of
Lemma~\ref{compprolo} simplifies to  $A= \frac{r}{R^d} \Big( \tfrac{1}{2} - \delta  -\nu -12 \delta \kappa
\Big)$. 
Our assumptions on  $\delta$ and $\nu$  guarantee that $A>0$.
\end{proof}

The formulation of Theorem~\ref{main0}  now follows.   With  $N =
R^d$ and $0< \nu < 1/2-\delta <1/2$,  if $\epsilon>0$ is given and
      \begin{equation}
        \label{eq:hm6}
        r \geq \max \Big(R^d \frac{1+\nu /3}{\nu ^2} \log
        \frac{2R^d}{\epsilon} \, , \frac{R^d}{3\log 3-2} \log
        \frac{2 (R+2)^d}{\epsilon}\Big) = R^d \frac{1+\nu /3}{\nu ^2} \log
        \frac{2R^d}{\epsilon} \, ,
      \end{equation}
then the probability in  \eqref{eq:hm5} will be larger than $1-\epsilon$.

\rem\ Observe that the parameters $\delta $ and $R$ are not
independent. As mentioned in ~\cite[p.~14]{BG10}, for $\brd $ to be
non-empty, we need $\delta \geq 2\pi \sqrt{2R} e^{-\pi R}$ (up to
terms of higher order). Thus for small $\delta $ as in Theorem~\ref{main} we need
to choose $R$ of order $R\approx c \log (d/\delta )$. 
 
\appendix

\section{The Plancherel-Polya inequality}

We finish by showing that the constant $\kappa $ in the Plancherel-Polya
inequality~\eqref{eq:heu1} can be chosen explicitly to be $\kappa = e^{d
  \pi }$.  The argument is simple and  well-known, see, for
example, \cite{Gro92b}. 
\begin{lemma}
  Let  $\{x_j: j\in \bN \} $ be a set in $\rd $ with
covering index $N_0$. Then 
$$  
\sum _{j=1}^\infty |f(x_j)|^2 \leq N_0 e^{d\pi } \|f\|_2^2 \, .
$$
\end{lemma}

\begin{proof}
 Let $k\in \zd $ and $x_j \in k+[-1/2,1/2]=: D_k$.  Then $\|x_j - k
  \|_\infty \leq 1/2$. Consider  the Taylor expansion of
  $f(x_j)$ at $k$ (with the usual multi-index notation):
$$
    |f(x_j)| = \big| \sum _{\alpha \geq 0} \frac{D^\alpha f(k)}{\alpha
    !} (x_j - k)^\alpha \big|  
\leq \sum _{\alpha \geq 0} \frac{|D^\alpha f(k)|}{\alpha
    !} \Big( \tfrac{1}{2} \Big)^{|\alpha |} \, .
$$
We now let $\theta\in (0,1)$ and apply Cauchy-Schwarz:
\begin{align}
|f(x_j)|^2 &\leq \sum _{\alpha \geq 0} \frac{1}{\alpha !}
\big(\tfrac{1}{2} \big)^{2  \theta | \alpha|} \, \, \sum _{\alpha \geq 0}
\frac{|D^\alpha f(k)|^2}{\alpha !} \big(\tfrac{1}{2} \big)^{2 (1-\theta
  )|\alpha|}
  \label{eq:c25}\\
&= e^{d/4^\theta } \sum _{\alpha \geq 0}
\frac{|D^\alpha f(k)|^2}{\alpha !} \big(\tfrac{1}{2} \big)^{2(1-\theta
  ) |\alpha|}. \nonumber
\end{align}

If $f\in \cB $, then by Shannon's sampling theorem (or because the
reproducing kernels $T_ks, k\in \zd $, form an \onb\ of $\cB $) we have 
$$
\sum _{k\in \zd } |f(k)|^2 = \|f\|_2 ^2 \qquad \forall f\in \cB  \, .
$$
To estimate the partial derivatives we use Bernstein's inequality $\|D^\alpha
f \|_2 \leq \pi ^{|\alpha |} \|f\|_2$.

We first assume that $N_0 = 1$, i.e., each cube $D_k$ contains at most
one of the $x_j$'s.  Then we obtain, after interchanging the order of summation 
\begin{align}
  \sum _{j\in \bN } |f(x_j)^2| &\leq    e^{d/4^\theta } \sum _{\alpha \geq 0} \sum _{k\in \zd } 
\frac{|D^\alpha f(k)|^2}{\alpha !} \big(\tfrac{1}{2}
\big)^{2(1-\theta ) |\alpha|} \notag  \\
& =   e^{d/4^\theta } \, \sum _{\alpha \geq 0} \big(\tfrac{1}{2}
\big)^{2(1-\theta ) |\alpha|} \frac{\|D^\alpha f\|_2^2}{\alpha !} \notag  \\
&\leq      e^{d/4^\theta }  \sum _{\alpha \geq 0}   \big(\tfrac{1}{2}
\big)^{2(1-\theta ) |\alpha|} \frac{\pi ^{2|\alpha |}
}{\alpha !}   \|f\|_2^2  = e^{d/4^\theta }  e^{d \pi ^2/4^{1-\theta }} \|f\|_2^2
\label{eq:c24}
\end{align}

  The choice $4^\theta = 2/\pi  $ yields the constant $\kappa =  e^{d/4^\theta }
 \, e^{d \pi ^2/4^{1-\theta }} = e^{d \pi }$. For arbitrary $N_0$ we
 obtain   
 \begin{align*}
  \sum _{j\in \bN } |f(x_j)^2|  = \sum _{k\in \zd }\  \sum _{\{j: x_j \in
    D_k\}} |f(x_j)^2| \leq N_0 e^{d \pi }  \|f\|_2^2 \, ,
\end{align*}
 as claimed. 
\end{proof}

Possibly the Plancherel-Polya inequality could be improved to a local
estimate of the form $\sum _{x_j \in C_R} |f(x_j|^2 \leq \tilde \kappa
N_0 \|f\|_{2,R}^2$, but we did not pursue this question.

\def\cprime{$'$} \def\cprime{$'$} \def\cprime{$'$} \def\cprime{$'$}
  \def\cprime{$'$}

 \bibliographystyle{abbrv}
 \bibliography{general,new}

\end{document}